\documentclass{amsart}

\usepackage{amssymb} 
\usepackage[utf8]{inputenc}
\usepackage[english]{babel}
\usepackage[T1]{fontenc}
\usepackage{graphicx}

\newtheorem{twierdzenie}{Theorem}[section] 
 
\newtheorem{wniosek}[twierdzenie]{Corollary}
\newtheorem{propozycja}[twierdzenie]{Proposition}
\theoremstyle{definition}
\newtheorem{definicja}[twierdzenie]{Definition}

\newtheorem{uwaga}[twierdzenie]{Remark}

\title{Generalized Lagrange Theorem}
\author{Karolina Zając}
\address{Karolina Zając - Jagiellonian University, Faculty of Mathematics and Computer Science, Institute of Mathematics, \L ojasiewicza 6, 30-348 Krak\'ow, Poland.}
\email{karolina99.zajac@student.uj.edu.pl}

\date{\today}

\begin{document}


\begin{abstract}

The present paper is devoted to possible generalizations of the classic Lagrange Mean Value Theorem. We consider a real-valued function of several variables that is only assumed to be continuous. The main concept is to replace the notion of the derivative by the so called bisequential tangent cone. We first prove Rolle and Lagrange type results and then we turn to comparing this cone with the Clarke subdifferential in the case of a Lipschitz function. We also investigate an approach using normal cones.
    
\end{abstract}
\maketitle

\section*{Introduction}

Lagrange's Mean Value Theorem in its classic form, for a differentiable single valued real function, is one of the most crucial facts in mathematical analysis, having a large number of important applications. 

A natural and interesting question is whether the Mean Value Theorem could have a valid counterpart for a function that is only continuous and not necessarily differentiable. In order to solve this problem, one has to find what a replacement for the derivative of the function or the tangent space of its graph.

There are many works concerning the above topic. A very interesting result, employing the generalized gradient, was published by G. Lebourg in \cite{lebourg}. Another approach, via approximate Jacobian matrices, can be found in \cite{j-l}. Some results using Dini derivatives were achieved by T. Ważewski and W. Mlak in \cite{wazewski} and \cite{mlak} respectively. Some possible generalizations can be also found in \cite{d-v}. For vector-valued maps, see also \cite{h-u} and references therein.

In the present paper yet another approach for multi-valued real continuous functions will be presented, using the paratingent cone, named here {\it the bisequential tangent cone (BTC)} (Definition \ref{dwustozek}). The BTC was first introduced by F. Severi and G. Bouligand respectively in the articles \cite{severi} and \cite{bouligand} and then studied in detail by Yuntong W. in \cite{yuntong}. It was also mentioned in several other sources such as \cite{r-w} or \cite{whitney}, with little to no further information provided. In the first section a handful of basic examples and main useful properties of this object will be presented and later a generalized version of the Mean Value Theorem will be formulated, namely a Rolle type result (Theorem \ref{rolle}) and a Lagrange type result (Theorem \ref{lagrange}). 

The second idea is to explore the properties of the normal cone. In this case the geometric interpretation is quite different. Instead of having a tangent parallel to the given secant, one can find in the mean point a nonzero vector from the normal cone perpendicular to the secant, see Theorem \ref{lagrange2}.

Finally, if the considered functions are restricted to be locally Lipschitz, the Clarke subdifferential can be used as the natural counterpart of the derivative. In the third section we recall the theorem by G. Lebourg and later we show a~correspondence between the Clarke subdifferential and the BTC in Theorem \ref{inkluzja}. This results in a very elegant geometric interpretation of Lebourg's result.

I would like to thank my supervisor Maciej Denkowski for suggesting the problem.

\section{Bisequential tangent cones}

Let us recall the basic Mean Value Theorem:

\begin{twierdzenie}[\sc Lagrange]
Let $a,b \in \mathbb{R}, a<b$ and let $f:[a,b] \rightarrow \mathbb{R}$ be a~continuous function, differentiable on an open interval $(a,b)$. Then
$$
\exists c\in(a,b): f'(c) = \frac{f(b)-f(a)}{b-a}.
$$
\end{twierdzenie}

As we would like to obtain a natural counterpart of the above theorem in the case where the function is only continuous, we may replace the derivative of the function by a properly understood tangent to the graph.

The most intuitive idea is to consider the Peano tangent cone:

\begin{definicja}
\label{styczny}
A vector $v\in\mathbb{R}^n$ is called {\it tangent to $S$ at $a$} if there exist sequences $\lbrace s_n \rbrace_{n=1}^{+\infty} \subset S$ and $\lbrace t_n \rbrace_{n=1}^{+\infty} \subset \mathbb{R}_+$ such that:

\begin{enumerate}
    \item $s_n\rightarrow a\ (n\rightarrow +\infty),$
    \item $t_n(s_n-a)\rightarrow v\ (n\rightarrow +\infty).$
\end{enumerate}

The set $C_a(S) = \lbrace v\in \mathbb{R}^n$ is tangent to $S$ at $a\rbrace$ is called the {\it (Peano) tangent cone of $S$ at $a$}.
\end{definicja}

Unfortunately, the above notion turns out to be insufficient. Take for example
$$ 
f:[-1,1]\ni x\mapsto \vert x\vert \in\mathbb{R}. 
$$
We see that the horizontal line $y = 0$ is not contained in the tangent cone of the graph at any of its points, which gives us a counterexample. So it will be necessary to take a slight modification of the classic definition. Thus we define the bisequential tangent cone (BTC) of a given set $X$ at a point $x \in X$ (see e.g. \cite{whitney}).

\begin{definicja}
\label{dwustozek}
We say that a vector $v\in\mathbb{R}^n$ belongs to the {\it bisequential tangent cone of $X$ at $x$} (we write $v\in B_x(X)$) iff there exists a bisequence of points 
$\lbrace (a_n,b_n) \rbrace_{n=1}^{+\infty} \subset X^2$ 
and a sequence of positive numbers 
$\lbrace t_n \rbrace_{n=1}^{+\infty} \subset \mathbb{R}_+$ such that:

\begin{enumerate}
    \item $a_n,b_n\rightarrow x\ (n\rightarrow +\infty),$
    \item $t_n(a_n-b_n)\rightarrow v\ (n\rightarrow +\infty).$
\end{enumerate}

Equivalently, if $\Vert v\Vert=1$, then $v \in B_x(X)$ if there exists a bisequence of points $\lbrace (a_n,b_n) \rbrace_{n=1}^{+\infty} \subset X^2$ for which
$$ \frac{a_n-b_n}{\Vert a_n-b_n\Vert} \rightarrow v, (n\rightarrow +\infty).$$
\end{definicja}

The above definition has a clear geometric interpretation, as the lines tangent to $X$ at $x$ are described by the limits of directions of secants with both ends approaching $x$.

Now let us examine some properties of the BTC.

\begin{uwaga}
There is always $0\in B_x(X)$. We also notice that unless $x$ is isolated in $X$, then $B_x(X) \neq \lbrace 0\rbrace$ and if $v \in B_x(X)$, then the BTC contains a whole line $\lbrace tv: t\in\mathbb{R} \rbrace$ and thus it is a full cone.
\end{uwaga}

In what follows we will assume that a (non-vertical) linear subspace of codimension $1$ can be identified with a~graph of a linear function $L:\mathbb{R}^{n-1} \rightarrow \mathbb{R}$.

\begin{definicja}
We say that a linear subspace as above is tangent to a set $X\in\mathbb{R}^n$ at $a\in X$ if
$(x,L(x))\in B_a(X), \forall x\in\mathbb{R}^{n-1}.$
\end{definicja}

From now on, we will consider the set $X$ being a graph of a~continuous function $f:A \rightarrow \mathbb{R}, A\subset \mathbb{R}^n$, denoted by $\Gamma_f$. For a linear map 
$$
L: \mathbb{R}^n \ni(x_1,\dots,x_n) \mapsto \alpha_1 x_1 + \dots + \alpha_n x_n \in \mathbb{R}
$$
such that $\Gamma_L \subset B_{(a,f(a))} (\Gamma_f)$ we will use simplified notations $B_{a, f(a)}(\Gamma_f) =: B_a(f)$ and $(\alpha_1, \dots, \alpha_n) \in \widehat{B}_a(f)$ or $L\in \widehat{B}_a(f)$, where $\widehat{B}_a(f)$ denotes the set of directions of tangent subspaces of $\Gamma_f \subset \mathbb{R}^n$ at the given point $a$, called the {\it BTC subdifferential of f at a} (see also \cite{yuntong} def. 4.1). 

Later a proper explanation of this name will be provided, as we will compare the BTC and Clarke subdifferentials of a Lipschitz function (see Theorem \ref{inkluzja}).

\begin{uwaga}
We have also to deal with a vertical tangent subspace which is a~hypersurface in $\mathbb{R}^{n+1} = \mathbb{R}^{n} \times \mathbb{R}$ given by
$$
\alpha_1 x_1 +\dots + \alpha_n x_n = 0,\ \alpha_1^2 + \dots + \alpha_n^2 \neq 0.
$$
If the BTC $B_a(f)$ contains such a hypersurface, we will write $\infty\in \widehat{B}_a(f)$.
\end{uwaga}

{\sc Basic examples}

\begin{enumerate}
    \item $f:A \rightarrow \mathbb{R}$ is a constant function, $a\in A$. Then $\widehat{B}_a(f) = \lbrace 0 \rbrace$. Later we will see that the set of tangent directions is reduced to $\lbrace 0 \rbrace$ in $a\in A$ if the function $f$ is of class $C^1$ in some neighbourhood of $a$ and $f'(a)=0$.
    
    \item $f:\mathbb{R}^n \rightarrow \mathbb{R}$ is a linear function, $A\subset \mathbb{R}^n, g = f\mid_A, a\in A$. Then we have $B_a(g) = \lbrace \Gamma_f \rbrace$.
    
    \item $A \subset B \subset \mathbb{R}^n$, $f: A \rightarrow \mathbb{R}, g: B \rightarrow \mathbb{R}$, $f = g\mid_A$, $a\in A$. Then we have $B_a(f) \subset B_a(g)$.
    
    \item If $f: A \rightarrow \mathbb{R}$ is differentiable at $a\in int A$, then always $\nabla f(a) \in \widehat{B}_a(f)$. It may happen however, that the BTC subdifferential will not be reduced to $\nabla f(a)$, see Example \ref{przyklad}.
    
    \item $f:\mathbb{R} \ni x\mapsto \vert x\vert \in\mathbb{R}$. Taking a bisequence
    $$
    \left\lbrace \left( \left( \frac{1}{n},\frac{1}{n}\right) ,\left( \frac{-t}{n}, \frac{t}{n}\right) \right)\right\rbrace_{n=1}^{+\infty}
    $$
    for some $t\in \mathbb{R}_{\geq 0}$, we obtain $\frac{1-t}{1+t} \in \widehat{B}_0(f)$, whence $\widehat{B}_0(f) = [-1,1]$.
    
\item $f: \mathbb{R} \ni x\mapsto x^{\frac{2}{3}} \in \mathbb{R}$. Take the map
$$
g: \mathbb{R} \ni x\mapsto x^{\frac{2}{3}} - \alpha x \in \mathbb{R}, \alpha \in \mathbb{R}
$$
and a bisequence
$$
\lbrace ((a_n,f(a_n)), (b_n,f(b_n))) \rbrace_{n=1}^{+\infty} \subset \mathbb{R}^2,
$$
such that $a_n,b_n \rightarrow 0\ (n\rightarrow +\infty)$ and $g(a_n) = g(b_n), a_n\neq b_n\ \forall n \in \mathbb{N}$. Then
$$
\frac{f(a_n) - f(b_n)}{a_n - b_n} = \frac{a_n^{\frac{2}{3}} - b_n^{\frac{2}{3}}}{a_n - b_n} = \frac{(a_n^{\frac{2}{3}} - \alpha a_n) - (b_n^{\frac{2}{3}} - \alpha b_n)+\alpha a_n - \alpha b_n}{a_n - b_n} = \alpha,
$$
so $\alpha \in \widehat{B}_0(f)$. But since $\alpha$ is chosen arbitrarily, then $\mathbb{R}\subset \widehat{B}_0(f)$. Moreover, by taking a bisequence
$$
\left\lbrace \left( \left( \frac{1}{2^{3n}}, \frac{1}{2^{2n}}\right) ,\left( \frac{1}{2^{3n+3}}, \frac{1}{2^{2n+2}}\right) \right) \right\rbrace_{n=1}^{+\infty},
$$
we obtain
$$
\frac{\frac{1}{2^{2n}} - \frac{1}{2^{2n+2}}}{\frac{1}{2^{3n}} - \frac{1}{2^{3n+3}}} = \frac{2^{n+3} - 2^{n+1}}{2^3 - 1} = 2^n \cdot \frac{6}{7} \rightarrow \infty\ (n\rightarrow +\infty).
$$
So $\widehat{B}_0(f) = \mathbb{R} \cup \lbrace \infty \rbrace$. The BTC equals the full space: $B_0(f) = \mathbb{R}^2$.
\end{enumerate}

\begin{uwaga}
In the case of a single variable function $f:A \rightarrow \mathbb{R}$ we may consider $\widehat{B}_a(f)$ for $a \in A$ as the set of limits $\lim_{n\rightarrow +\infty} \frac{f(a_n) - f(b_n)}{a_n - b_n}$, where $$\lbrace (a_n,b_n) \rbrace_{n=1}^{+\infty} \subset A^2, a_n, b_n \rightarrow x\ (n\rightarrow +\infty), a_n \neq b_n, \forall n\in\mathbb{N}.$$
\end{uwaga}

\begin{uwaga}
\label{przyklad}
Differentiability of $f$ at a point $a$ is not a sufficient condition for $\widehat{B}_a(f) = \lbrace \nabla f(a)\rbrace$. Consider a well-known example
$$
f:\mathbb{R}\ni x\mapsto
\begin{cases}
x^2 \sin (\frac{1}{x}), & x \in \mathbb{R} \setminus \lbrace 0\rbrace\\
0, & x=0
\end{cases}
\in\mathbb{R}.
$$
The function is differentiable at $0$ and $f'(0) = 0$, yet the derivative is not continuous at the origin. Thus we obtain $\widehat{B}_0(f)=[-1,1]$.
\end{uwaga}

In order to find any applications for the BTC subdifferential, some basic properties are needed. The following proposition is well-known for Clarke subdifferential (see for example \cite{r-w}, Theorem 9.8, or \cite{yuntong}, Proposition 4.3). We provide a direct proof based on the definition of the BTC subdifferential.

\begin{propozycja}
\label{nwsr}
If $A$ is open in $\mathbb{R}^n$ and $f:A\rightarrow\mathbb{R}$ is differentiable, then for any point $a\in A$, the following conditions are equivalent:

\begin{enumerate}
    \item $f$ is of class $C^1$ in some neighbourhood of $a$,
    \item $\widehat{B}_a(f) = \lbrace \nabla f(a)\rbrace$.
\end{enumerate}
\end{propozycja}

\begin{proof}
$\ $

(1) $\Rightarrow$ (2) 
Assume that there exists $L \in \mathcal{L} (\mathbb{R}^n; \mathbb{R})$, $L \neq f'(a)$ for which we have $\nabla L \in \widehat{B}_a(f)$. Pick $b\in\mathbb{R}^n$ such that $L.b \neq f'(a).b$ and set $\varepsilon > 0$ for which $\vert L.b - f'(a).b \vert > 3 \varepsilon$. Then choose a bisequence $\lbrace (a_n,b_n) \rbrace_{n=1}^{+\infty}\subset  A^2$ such that $a_n,b_n \rightarrow a\ (n\rightarrow +\infty)$ and
$$
t_n(a_n-b_n,f(a_n)-f(b_n))\rightarrow (b,L.b)\ (n\rightarrow +\infty)
$$
for some sequence $\lbrace t_n \rbrace_{n=1}^{+\infty} \subset \mathbb{R}$. Then there exist $N_1,N_2 \in\mathbb{N}$ such that
\begin{align*}
    \vert t_n(f(a_n)-f(b_n)) - L.b \vert < \varepsilon, \quad & \forall n \geq N_1,\\
    \vert L.(t_n(a_n-b_n)) - L.b \vert < \varepsilon, \quad & \forall n \geq N_2.
\end{align*}
Let $N=max\lbrace N_1,N_2\rbrace$. We notice that for $n > N$
\begin{align*}
& \vert t_n(f(a_n)-f(b_n)) - L.(t_n(a_n-b_n)) \vert\\
\leq\ & \vert t_n(f(a_n)-f(b_n)) - L.b \vert + \vert L.(t_n(a_n-b_n)) - L.b \vert < 2 \varepsilon.
\end{align*}
The function $f\mid_{[a_n,b_n]}$ for $n>N$ can be identified by a projection with a single variable function and the basic Mean Value Theorem can be applied. We find a~point $c_n$ on the segment $[a_n,b_n]$ for which
$$
f'(c_n).(t_n(a_n-b_n)) = (f\mid_{[a_n,b_n]})'(c_n).(t_n(a_n-b_n)) = t_n(f(a_n)-f(b_n)).
$$
Then we obtain
\begin{align*}
& \vert f'(c_n).(t_n(a_n-b_n)) - f'(a).(t_n(a_n-b_n)) \vert\\ 
=\ & \vert t_n(f(a_n)-f(b_n)) - f'(a).(t_n(a_n-b_n)) \vert\\
\geq\ & \vert L.(t_n(a_n-b_n)) - f'(a).(t_n(a_n-b_n)) \vert\\
& - \vert t_n(f(a_n)-f(b_n)) - L.(t_n(a_n-b_n)) \vert\\
>\ & \vert L.(t_n(a_n-b_n)) - f'(a).(t_n(a_n-b_n)) \vert - 2 \varepsilon.
\end{align*}
Taking a limit of both sides when $n \rightarrow +\infty$, we have
$$
0 = \vert f'(a).b - f'(a).b \vert \geq \vert L.b - f'(a).b \vert - 2 \varepsilon > 3 \varepsilon - 2 \varepsilon = \varepsilon,
$$
which yields a contradiction.

$\ $

(2) $\Rightarrow$ (1) Fix $\varepsilon >0$. Then, by assumptions, the directions of secants in $\Gamma_f$ with both ends within a tight neighbourhood of $(a, f(a))$ are close to the direction determined by the derivative $f'(a)$. In other words:
$$
\exists r_a>0: (\Vert x-a \Vert < r_a, \Vert y-a\Vert < r_a) \Rightarrow (\vert(f(x)-f(y))-f'(a).(x-y)\vert\leq\varepsilon)
$$
for $x,y \in A$. Now set $b\in A, \vert b-a\vert <r_a$ and define $r_b$ analogically. Then for $x,y \in \mathbb{B}(a,r_a)\cap \mathbb{B}(b,r_b)$ we have:
\begin{multline*}
\vert f'(b).(x-y)-f'(a).(x-y)\vert\leq\\
\vert f'(b).(x-y)-(f(x)-f(y))\vert + \vert(f(x)-f(y))-f'(a).(x-y)\vert\leq 2\varepsilon,
\end{multline*}
which gives the continuity of $f'$, so $f$ is of class $C^1$ near $a$.
\end{proof}

Another useful property refers to the behaviour of the BTC subdifferential of a~sum of two functions (cf. \cite{r-w}, exercise 8.8).

\begin{propozycja}
\label{sum}
Let $A \subset \mathbb{R}^n$, $a \in A$. For two continuous functions $f,g :A \rightarrow \mathbb{R}$ assume that $g$ is of class $C^1$ near $a$. Then $\widehat{B}_a(f+g) = \widehat{B}_a(f) + \nabla g(a)$.
\end{propozycja}

\begin{proof}
The inclusion "$\supset$" can be obtained directly (use \ref{nwsr}). Then we get "$\subset$" by applying the above rule to $f+g$ and $-g$.
\end{proof}

\begin{uwaga}
We necessarily have to assume that one of the functions is of class $C^1$. Consider as an example
$$
f:[-1,1]\ni x\mapsto \vert x\vert \in\mathbb{R},
$$
$$
g:[-1,1]\ni x\mapsto -\vert x\vert \in\mathbb{R}.
$$
We see that $1\in \widehat{B}_0(f)$ and $1\in \widehat{B}_0(g)$. However, $2 \notin \widehat{B}_0(f+g)=\lbrace 0\rbrace$.
\end{uwaga}

Now we state generalized Rolle and Lagrange Theorems for the BTC.

\begin{twierdzenie}[\sc Generalized Rolle]
\label{rolle}
Let $K$ be a compact subset of $\mathbb{R}^n$ such that $int K \neq \emptyset$ and let $f:K\rightarrow\mathbb{R}$ be a continuous function, $f \equiv C$ for some $C \in \mathbb{R}$ on $\partial K$. Then
$$
\exists c\in int K:\ 0\in \widehat{B}_c(f).
$$
\end{twierdzenie}

\begin{proof}
If $f$ is constant, then the theorem obviously holds. Otherwise, $f$ reaches two different extrema in $K$ by Weierstrass Theorem. Without any loss of generality, we may assume that
$$
\exists c\in int K:f(c)=\inf_{x\in K}f(x),
$$
We consider two cases.

\begin{enumerate}
    \item Case $n=1$.
    
    We consider separately the situation in which there exists a sequence $\lbrace c_n \rbrace_{n=1}^{+\infty} \subset K \setminus \lbrace c \rbrace$ of pairwise different points, convergent to $c$, such that $f(c_n) = f(c), \forall n\in \mathbb{N}$. Then a bisequence $\lbrace (c_n,c_{n+1}) \rbrace_{n=1}^{+\infty}$ together with a~sequence of numbers $\lbrace \frac{1}{c_n-c_{n+1}} \rbrace_{n=1}^{+\infty}$ is sufficient to obtain $(1,0) \in B_c(f)$.
    
    Otherwise, since $f$ is continuous, by the Darboux Theorem we are able to construct a bisequence $\lbrace (a_n,b_n) \rbrace_{n=1}^{+\infty} \subset K^2$ such that 
    $$
    a_n\rightarrow c_+,b_n\rightarrow c_- (n\rightarrow+\infty),
    $$
    $$
    f(a_n) = f(b_n) = f(c) + \frac{C-f(c)}{n}, \forall n\in\mathbb{N}.
    $$
    Taking the sequence of numbers $\lbrace \frac{1}{a_n-b_n} \rbrace_{n=1}^{+\infty}$, we obtain
    $$
    \frac{1}{a_n-b_n} (a_n-b_n, f(a_n)-f(b_n)) = (1,0).
    $$

\item Case $n>1$.

Fix $x\in\mathbb{R}^n$. We will show that $(x,0)\in B_c(f)$. We choose $a,b \in \partial K$ in such a way that $c\in[a,b]\subset K$ and the point $c+x$ belongs to a line going through $a$ and $b$. Then we can identify $f\mid_{[a,b]}$ with a single variable function and the rest of the proof is done as in the previous case with the sequence of numbers multiplied by $\pm\Vert x\Vert$.
\end{enumerate}
\end{proof}

\begin{twierdzenie}[\sc Generalized Lagrange]
\label{lagrange}
Let $K$ be a compact subset of $\mathbb{R}^n$, $int K \neq \emptyset$ and let $f:K\rightarrow\mathbb{R}$ be a continuous function, $f(x) = L.x+C, \forall x  \in \partial K$, for some $L \in \mathcal{L} (\mathbb{R}^n; \mathbb{R}), C\in\mathbb{R}$. Then
$$
\exists\ c\in int K:\ L\in \widehat{B}_c(f).
$$
\end{twierdzenie}

\begin{proof}
Let us consider the map $g:K\ni x\mapsto f(x)-L.x\in\mathbb{R}$. It is constant on $\partial K$, so by \ref{rolle} $\exists c\in int K:0\in \widehat{B}_c(f)$. Obviously $\widehat{B}_c(L) = \lbrace \nabla L\rbrace$ and \ref{sum} applied for $g$ and $L$ ends the proof.
\end{proof}

\begin{uwaga}
It is necessary to assume that $K$ is compact. Consider as an example the function $f(x,y) = x^2y$ and $K = \lbrace (x,y)\in \mathbb{R}^2: -1 \leq x \leq 1 \rbrace$. We have $f(x,y) = y$ on $\partial K$ but $\widehat{B}_{(x,y)}(f) = (2xy, x^2) \neq (0,1)$ inside $K$.
\end{uwaga}

\section{Normal cones}

Let $L$ be a linear subspace of $\mathbb{R}^n$. Any vector $v\in\mathbb{R}^n$ can be uniquely presented in a form $v=v_1+v_2$, where $v_1\in L,v_2\in L^{\perp}$ and $\Vert v-v_1\Vert=\inf\lbrace\Vert v-w\Vert:w\in L\rbrace$. The map
$ P_L:\mathbb{R}^n\ni v\mapsto v_1\in L $
is the orthogonal projection on a subspace $L$.

Now we consider the following situation. Let $S\subset\mathbb{R}^n$. For $a\in S$ we recall the normal cone of $S$ at $a$.

\begin{definicja}
The {\it normal cone of $S$ at $a$} is the set
$$
N_a(S)=\lbrace v\in\mathbb{R}^n:\langle v,w\rangle\leq 0, \forall w\in C_a(S)\rbrace.
$$
\end{definicja}

If our considered set $S$ is a graph of a function $f:A \rightarrow \mathbb{R}$, $A \subset \mathbb{R}^{n-1}$, then for $a\in A$ we will use the following notation:
$$
C_a(f) := C_{(a,f(a))}(\Gamma_f),\ N_a(f) = N_{(a,f(a))}(\Gamma_f).
$$

\begin{uwaga}
If the function $f:A \rightarrow \mathbb{R}$, $A \subset \mathbb{R}^{n-1}$ is differentiable at $a\in int A$, then $C_a(f) = \Gamma_{f'(a)}$ and $N_a(f) = (\Gamma_{f'(a)})^{\perp}$.
\end{uwaga}

For the sake of further considerations, we introduce a modified definition of an angle between two (non-zero) vectors.

\begin{definicja}
For $x,y \in\mathbb{R}^n$, $x,y\neq 0$ we define the {\it angle between $x$ and $y$} as
$$
\angle(x,y) = \frac{\langle x,y\rangle}{\Vert x\Vert\Vert y\Vert} \in[-1,1].
$$
\end{definicja}

\begin{uwaga}
By the above definition, the second condition in \ref{styczny} coincides with $\angle(s_n-a,v)\rightarrow 1\ (n\rightarrow +\infty)$. The set $N_a(S)$ is $\lbrace v\in\mathbb{R}^n: \angle(v,w) \leq 0, \forall w \in C_a(S) \rbrace$.
\end{uwaga}

We will show the Mean Value Theorem stated for a continuous function using the normal cone. The following known result will be crucial.

\begin{propozycja}
\label{char}
Let $a\in S\subset\mathbb{R}^n$ and let $x\in\mathbb{R}^n$ be a point for which the Euclidean distance from $S$ is realized at $a$, that is
$$
\Vert x-a\Vert=\inf\lbrace\Vert x-y\Vert:y\in S\rbrace.
$$
Then $x-a\in N_a(S)$.
\end{propozycja}

\begin{proof}
We may assume that $a=0$. If $x \notin N_0(S)$, then by definition there exists a~point $v\in C_0(S), v\neq 0$ for which $\langle x,v\rangle>0$. Without loss of generality we assume that $v \in \partial \mathbb{B}(x, \Vert x\Vert)$. According to \ref{styczny},
$$
\exists \lbrace (a_n) \rbrace_{n=1}^{+\infty} \subset S^2,
\lbrace t_n \rbrace_{n=1}^{+\infty} \subset \mathbb{R}_+:
a_n \rightarrow 0\ (n\rightarrow +\infty),
t_n a_n \rightarrow v\ (n\rightarrow +\infty).
$$
Then $\angle(a_n,x) \rightarrow \angle(v,x)>0\ (n\rightarrow +\infty)$, so
$$
\exists N \in \mathbb{N}: \angle(a_n,x)>0, \forall n\geq N,
$$
This can be interpreted geometrically as the statement that the lines going through $0$ and $a_n$ intersect $\partial \mathbb{B}(x, \Vert x\Vert)$ at some points $b_n \neq 0$ for all $n\geq N$. We see that $b_n\rightarrow v\ (n\rightarrow +\infty)$, so for $n$ sufficiently large, the point $a_n$ must be inside the segment $[0,b_n]$. Otherwise we would have
$$
0 = \lim_{n\rightarrow +\infty} \Vert a_n \Vert \geq \lim_{n \rightarrow +\infty} \Vert b_n \Vert = \Vert v\Vert>0.
$$
But then $\Vert x-a_n \Vert < \Vert x \Vert$ for some $n \in \mathbb{N}$. This contradiction finishes the proof.
\end{proof}

\begin{uwaga}
The converse implication is not always true. Take as an example the graph of $f(x) = \vert \vert x \vert - 1 \vert$. Then the point $(0,7)$ belongs to $N_0(f)$ but its Euclidean distance from the graph is not realized at $(0,1)$.
\end{uwaga}

Now to the main theorem.

\begin{twierdzenie}
\label{lagrange2}
Let $K$ be a compact subset of $\mathbb{R}^n$, $int K \neq \emptyset$ and let $f:K\rightarrow\mathbb{R}$ be a continuous function such that $f(x) = L.x+C, \forall x \in \partial K$ for some $L \in \mathcal{L} (\mathbb{R}^n; \mathbb{R})$, $C \in \mathbb{R}$. Then
$$
\exists c\in int K: (N_c(f)\cap (\Gamma_L)^{\perp}) \setminus \lbrace 0\rbrace \neq\emptyset.
$$
In another words, a normal cone at some point $c$ inside the set $K$ contains a non-zero vector perpendicular to $\Gamma_L$.
\end{twierdzenie}

\begin{proof}
Without a loss of generality, let $C = 0$. We consider the map
$$
g:K \ni x \mapsto \Vert P_{(\Gamma_L)^{\perp}}((x,f(x))) \Vert \in \mathbb{R}_{\geq 0}.
$$
It is continuous as a composition of continuous functions and defined on a compact set, so it reaches its limits. By the assumptions we know that $g \mid_{\partial K} \equiv 0$. If $g \equiv 0$, then $f$ is affine and $N_x(f) = L^{\perp}, \forall x\in int K$. Otherwise we choose a point $c\in int K$ such that $g(c) = \sup_{x\in K} g(x)$ and we consider the vector 
$$
v = P_{(\Gamma_L)^{\perp}} ((c,f(c))).
$$
Then $v$ is perpendicular to $\Gamma_L$ and moreover
$$
\forall x \in K: \Vert P_{(\Gamma_L)^{\perp}}((x,f(x))) \Vert \leq \Vert v \Vert. 
$$
We will show that $v\in N_c(f)$. It is enough to check that the Euclidean distance of $v+(c,f(c))$ from the graph of $f$ is realized at $(c,f(c))$. Let us fix $x\in K\setminus\lbrace c\rbrace$ and notice that
\begin{eqnarray*}
\Vert v+(c,f(c))-(x,f(x)) \Vert &\geq& \Vert P_{(\Gamma_L)^{\perp}}(v+(c,f(c))-(x,f(x))) \Vert\\
&=& \Vert P_{(\Gamma_L)^{\perp}}(v) + P_{(\Gamma_L)^{\perp}}((c,f(c))) - P_{(\Gamma_L)^{\perp}}(x,f(x))) \Vert\\
&=& \Vert 2v-P_{(\Gamma_L)^{\perp}}((x,f(x))) \Vert\\
&\geq& 2\Vert v\Vert - \Vert P_{(\Gamma_L)^{\perp}}((x,f(x))) \Vert \geq \Vert v\Vert.
\end{eqnarray*}
Then by Proposition \ref{char} we have $v\in N_c(f)$, what was to be proven.
\end{proof}

\section{Clarke subdifferential}

In this chapter we will consider the case where the function fulfills the local Lipschitz condition.

\begin{definicja}
\label{lipshitz}
Let $A \subset \mathbb{R}^n$. The function $f: A \rightarrow \mathbb{R}$ is called {\it locally Lipschitz} if for any point $a \in A$ there exist $M>0$, $\varepsilon >0$ such that
$$
\forall x,y \in A:\ (x,y \in K(a,\varepsilon)) \Rightarrow (\vert f(x)-f(y) \vert \leq M\Vert x-y \Vert).
$$
\end{definicja}

All further considerations in this section will be taken for a locally Lipschitz function $f$, defined as in \ref{lipshitz}. The word {\it locally} can be omitted, as the set $A \subset \mathbb{R}^n$ can be narrowed to a properly adjusted neighbourhood of the point $a \in A$.

For the convenience of the reader we will recall several basic notions from \cite{clarke}, first the generalized directional derivative of the function $f$ at the point $a$ in the direction $v \in \mathbb{R}^n$.

\begin{definicja}
Let $f:A \rightarrow \mathbb{R}$ be a Lipschitz function, $A \subset \mathbb{R}^n$. The {\it generalized directional derivative} of the function $f$ at the point $a\in A$ and in the direction $v\in \mathbb{R}^n$ is the limit
$$
f^{\circ}(a;v) = \limsup_{x\rightarrow a, t\rightarrow 0^+} \frac{f(x+tv) - f(x)}{t}.
$$
\end{definicja}

\begin{propozycja}
$f^{\circ}(a;v)$ as a function of a single variable $v$ is bounded, positively homogeneous and subadditive. Moreover,
    $$
    \vert f^{\circ}(a;v) \vert \leq M\Vert v \Vert\ \forall a\in A, v\in \mathbb{R}^n,
    $$
    where $M$ is the Lipschitz constant of $f$.
\end{propozycja}

We can apply the Hahn-Banach Theorem for the function $f^{\circ}(a;v)$, finding a~linear functional $\xi : \mathbb{R}^n \rightarrow \mathbb{R}$ such that
$$
f^{\circ}(a;v) \geq \xi(v), \forall a\in A, v\in \mathbb{R}^n.
$$

\begin{definicja}
The set
$
\partial f(a) = \lbrace \xi \in \mathbb{R}^n: f^{\circ}(a;v) \geq \langle \xi,v \rangle\ \forall v\in \mathbb{R}^n \rbrace
$
is called the {\it generalized gradient} of $f$ at $a\in A$.
\end{definicja}

Now let our attention be paid to the basic properties of the generalized gradient, given in \cite{clarke}.

\begin{propozycja}
\label{prop1}
For any $\lambda \in \mathbb{R}$ we have $\partial(\lambda f)(a) = \lambda \partial f(a)$.
\end{propozycja}

\begin{propozycja}
\label{prop2}
If the function $f$ reaches a local extreme at $a$, then $0 \in \partial f(a)$.
\end{propozycja}

\begin{propozycja}
\label{prop3}
For lipschitz functions $f_i: A_i \rightarrow \mathbb{R}$, $i= 1,\dots,n$, $A_i \subset \mathbb{R}^n$, $a\in \bigcap_{i=1}^n A_i$ we have
$$
\partial \left( \sum_{i=1}^n f_i \right)(a) \subset \sum_{i=1}^n \partial f_i(a),
$$
where the set on the right hand side consists of elements of the form 
$$
 \sum_{i=1}^n \xi_i: \xi_i \in \partial f_i,\ i=1,\dots,n .
$$
\end{propozycja}

We can see how the Mean Value Theorem can be formulated for a Lipschitz function using the generalized gradient. In \cite{clarke} we can find the following fact, known as the Lebourg Theorem.

\begin{twierdzenie}[\sc Lebourg]
Let $U \subset \mathbb{R}^n$ be an open set and let $x,y \in U$ be such that $[x,y] \subset U$. Then for a Lipschitz function $f:U \rightarrow \mathbb{R}$ there exists a point $c\in (x,y)$ such that $f(y)-f(x) \in \langle \partial f(c), y-x \rangle.$
\end{twierdzenie}

There are several other notions of a subdifferential. In some sources (see e.g. \cite{praca} or \cite{r-w}) the following can be found:

\begin{definicja}
Let $U$ be an open and nonempty subset of $\mathbb{R}^n$ and let $f:U\rightarrow\mathbb{R}$ be a locally lipschitz function. For $x\in U$ we define:

\begin{enumerate}
    \item The \textit{Fr\'{e}chet subdifferential} $\widehat{\partial}f(x)$ of $f$ at $x$ as
$$
\widehat{\partial}f(x) = \left\lbrace v\in\mathbb{R}^n: \liminf_{y\rightarrow x,y\neq x} \frac{f(y)-f(x)-\langle v,y-x\rangle}{\Vert y-x\Vert}\geq 0\right\rbrace,
$$

\item The \textit{limit subdifferential} $\partial' f(x)$ of $f$ at $x$ as
\begin{equation*}
\begin{split}
v\in\partial' f(x)\Leftrightarrow\ & \exists \lbrace x_n \rbrace_{n=1}^{+\infty} \subset U,\lbrace v_n \rbrace_{n=1}^{+\infty} \subset \mathbb{R}^n:\\
& x_n\rightarrow x,v_n\rightarrow v\ (n\rightarrow +\infty),v_n\in\widehat{\partial}f(x_n), \forall n\in\mathbb{N},
\end{split}
\end{equation*}

\item The \textit{Clarke subdifferential} $\partial^{\circ} f(x)$ of $f$ at $x$ as the convex hull of $\partial' f(x)$ ($\partial^{\circ} f(x) = cvx\ \partial' f(x)$).

\end{enumerate}
\end{definicja}

It follows immediately that $\widehat{\partial}f(x) \subset \partial' f(x) \subset \partial^{\circ} f(x)$. 

In \cite{r-w}, the links between all the different notions of subdifferentials are clearly explained. In particular, it is well-known that $\partial f(x) = \partial^{\circ} f(x)$ for $f$ being a locally Lipschitz function.

\begin{uwaga}
\label{uw}
The Fr\'{e}chet subdifferential has a convenient geometric interpretation. For $x \in U$ consider the tangent cone $C_x(f)$, taking into account all sequences $\lbrace x_n \rbrace_{n=1}^{+\infty} \subset U$ convergent to $x$ such that
$$
\exists (v,w) \in C_x(f): \frac{x_n-x}{\Vert x_n-x \Vert} \rightarrow v,\ \frac{f(x_n)-f(x)}{\Vert x_n-x \Vert} \rightarrow w\ (n \rightarrow +\infty)
$$
and $x_n \neq x\ \forall n \in \mathbb{Z}_+$. Then $\widehat{\partial}f(x)$ can be viewed as the set of vectors $\xi \in \mathbb{R}^n$ such that for any sequence $\lbrace x_n \rbrace_{n=1}^{+\infty}$ we have
$$
\lim_{n \rightarrow +\infty} \frac{f(x_n)-f(x)-\langle \xi, x_n-x\rangle}{\Vert x_n-x\Vert}\geq 0
$$
or $w \geq \langle \xi,v \rangle, \forall (v,w) \in C_x(f)$.
\end{uwaga}

Now we shall see the connection between the Clarke subdifferential and the BTC subdifferential considered in the first section.

\begin{twierdzenie}
\label{inkluzja}
Let $U$ be open in $\mathbb{R}^n$, $a\in U$ and let $f:U\rightarrow\mathbb{R}$ be a locally lipschitz function. Then we have $\partial^{\circ}f(a) = \widehat{B}_a(f).$
\end{twierdzenie}

\begin{proof}
Without loss of generality we may assume that $a = f(a) = 0$. Let us fix $v \in \partial^{\circ}f(a)$. First we will prove that
$$
(x, \langle v,x \rangle) \in B_a(f),\ \forall x \in \mathbb{R}^n,
$$
what will be done in three steps.

{\sc Step 1.} $v \in \widehat{\partial} f(0)$. 

Fix $x\in\mathbb{R}^n$, $x \neq 0$. Let $\Gamma_x$ denote the cross-section of the graph of the function above the segment $[0,x]$. We choose the smallest numbers $y_1,y_2\in\mathbb{R}$ such that
$$
(-x,y_1),(x,y_2)\in C_0(\Gamma_x) \subset C_0(f),
$$
which surely exist because of the Lipschitz condition.

If $y_1<-y_2$, then, by \ref{uw}, $\widehat{\partial}f(0) = \emptyset$ and there is nothing to prove. Assume that $y_1 \geq -y_2$. Then we find sequences $\lbrace a_n \rbrace_{n=1}^{+\infty} \subset (0,-x]$, $\lbrace b_n \rbrace_{n=1}^{+\infty} \subset (0,x]$ satisfying
$$
a_n\rightarrow 0,\ b_n\rightarrow 0\ (n\rightarrow +\infty),
$$
$$
\left\vert \frac{f(a_n)}{\Vert a_n\Vert} - \frac{y_1}{\Vert x\Vert} \right\vert \leq\frac{1}{2^n},\ \left\vert \frac{f(b_n)}{\Vert b_n\Vert} - \frac{y_2}{\Vert x\Vert}\right\vert \leq\frac{1}{2^n}, \forall n\in\mathbb{N}.
$$
Without loss of generality assume that $\vert a_n\vert,\vert b_n\vert<\frac{1}{2^n}$ for $n\in\mathbb{N}$, choosing proper subsequences if needed. Now we wish to construct a bisequence $\lbrace (c_n,d_n) \rbrace_{n=1}^{+\infty} \subset U^2$ such that
$$
c_n,d_n\rightarrow 0,\ \frac{f(c_n)-f(d_n)}{\Vert c_n-d_n \Vert}\rightarrow \frac{\langle v,x \rangle}{\Vert x\Vert}\ (n\rightarrow +\infty).
$$
Fix $N\in\mathbb{Z}_+$. If
$$
\left\vert \frac{f(a_N)-f(b_N)}{\Vert a_N-b_N \Vert} - \frac{\langle v,x \rangle}{\Vert x\Vert} \right\vert \leq\frac{1}{2^N},
$$
then we take $c_N=a_N, d_N=b_N$. Otherwise, if
$$
\frac{f(a_N)-f(b_N)}{\Vert a_N-b_N \Vert}< \frac{\langle v,x \rangle}{\Vert x\Vert} - \frac{1}{2^N},
$$ 
then we fix $c_N=a_N$ and consider the map
$$
g:U\ni y\mapsto \frac{f(a_N)-f(y)}{\Vert a_N-y \Vert}\in\mathbb{R}.
$$
It is continuous as a composition of continuous functions and furthermore 
$$
g(b_N)< \frac{\langle v,x \rangle}{\Vert x\Vert} -\frac{1}{2^N},\ g(0)\geq \frac{y_2}{\Vert x\Vert} -\frac{1}{2^N}\geq \frac{\langle v,x \rangle}{\Vert x\Vert} -\frac{1}{2^N}.
$$
So we can find a mean point $\xi\in[0,b_N]$ such that
$$
g(\xi)= \frac{\langle v,x \rangle}{\Vert x\Vert}-\frac{1}{2^N}
$$
and we choose $d_N=\xi$. If
$$
\frac{f(a_N)-f(b_N)}{\Vert a_N-b_N \Vert}> \frac{\langle v,x \rangle}{\Vert x\Vert}+\frac{1}{2^N},
$$
then we set $d_N=b_N$ and analogically we find a mean point $c_N\in[a_N,0]$ for which
$$
\frac{f(c_N)-f(d_N)}{\Vert c_N-d_N \Vert}= \frac{\langle v,x \rangle}{\Vert x\Vert}+\frac{1}{2^N}.
$$
It can be easily observed that this bisequence fulfills the desired conditions. Moreover, all segments $[c_i,d_i]$, $i\in\mathbb{Z}_+$ are parallel to the segment $[0,x]$.

{\sc Step 2.} $v\in\partial' f(0)$. 

Let $\lbrace x_n \rbrace_{n=1}^{+\infty} \subset U$ and $\lbrace v_n \rbrace_{n=1}^{+\infty} \subset \mathbb{R}^n$ be sequences such that:
$$
x_n\rightarrow 0\ (n\rightarrow +\infty),\ v_n\in\widehat{\partial}f(x_n)\ \forall n\in\mathbb{N},\ \Vert v_n-v\Vert \leq\frac{1}{2^n}, \forall n\in\mathbb{N}.
$$
Having fixed a point $x\in\mathbb{R}^n$, for all pairs $(x_i,v_i)$, $i\in\mathbb{Z}_+$ we construct bisequences $\lbrace (c_k^{(i)},d_k^{(i)}) \rbrace_{k=1}^{+\infty} \subset U^2$ as in the first step. Then for any $n\in\mathbb{N}$ we have
\begin{eqnarray*}
\left\vert \frac{f(c_n^{(n)})-f(d_n^{(n)})}{\Vert c_n^{(n)}-d_n^{(n)} \Vert} - \frac{\langle v,x \rangle}{\Vert x\Vert} \right\vert &\leq & \left\vert \frac{f(c_n^{(n)})-f(d_n^{(n)})}{\Vert c_n^{(n)}-d_n^{(n)} \Vert} - \frac{\langle v_n,x \rangle}{\Vert x\Vert} \right\vert + \left\vert \frac{\langle v_n,x \rangle}{\Vert x\Vert} - \frac{\langle v,x \rangle}{\Vert x\Vert} \right\vert\\
&\leq & \frac{1}{2^n}+\frac{1}{2^n} = \frac{1}{2^{n-1}}.
\end{eqnarray*}
Hence for the bisequence $\lbrace (c_n^{(n)},d_n^{(n)}) \rbrace_{n=1}^{+\infty}$ we have
$$
\frac{f(c_n^{(n)})-f(d_n^{(n)})}{\Vert c_n^{(n)}-d_n^{(n)} \Vert}\rightarrow \frac{\langle v,x \rangle}{\Vert x\Vert}\ (n\rightarrow +\infty),
$$
and $c_n^{(n)},d_n^{(n)}\rightarrow 0\ (n\rightarrow +\infty)$. Moreover, each of the segments $[c_i^{(i)},d_i^{(i)}]$, $i\in\mathbb{Z}_+$ is parallel to the segment $[0,x]$.

{\sc Step 3.} $v\in\partial^{\circ} f(0)$. 

If the set $\partial' f(0)$ is convex, then there is nothing to prove. Otherwise we choose a vector $v\in \partial^{\circ} f(0) \setminus \partial' f(0)$ and fix $v_1,v_2\in \partial' f(0)$ such that $v= T v_1+(1-T) v_2$ for some $T\in (0,1)$. Analogically as in the second step for a given $x\in\mathbb{R}^n$ we may construct bisequences $\lbrace (a_n,b_n) \rbrace_{n=1}^{+\infty}$, $\lbrace (c_n,d_n) \rbrace_{n=1}^{+\infty} \subset U^2$ such that
$$
a_n,b_n,c_n,d_n\rightarrow 0,\ \frac{f(a_n)-f(b_n)}{\Vert a_n-b_n \Vert}\rightarrow \frac{\langle v_1,x \rangle}{\Vert x\Vert},\ \frac{f(c_n)-f(d_n)}{\Vert c_n-d_n \Vert}\rightarrow \frac{\langle v_2,x \rangle}{\Vert x\Vert}\ (n\rightarrow +\infty).
$$
Now we may search for a bisequence $\lbrace (e_n,f_n) \rbrace_{n=1}^{+\infty} \subset U^2$ such that
$$
e_n,f_n\rightarrow 0,\ \frac{f(e_n)-f(f_n)}{\Vert e_n-f_n \Vert}\rightarrow \frac{\langle v,x \rangle}{\Vert x\Vert}\ (n\rightarrow +\infty).
$$
Without loss of generality we may assume that $\langle v_1,x \rangle < \langle v_2,x \rangle$. Then
$$
\exists\ N_0\in \mathbb{Z}_+: \frac{\langle v_1,x \rangle}{\Vert x \Vert} + \frac{1}{2^{N_0}} \leq \frac{\langle v,x \rangle}{\Vert x \Vert} \leq \frac{\langle v_2,x \rangle}{\Vert x \Vert} - \frac{1}{2^{N_0}}.
$$
For all $i\in\mathbb{Z}_+$ we consider a continuous mapping
$$
g_i:[0,1] \ni t \mapsto \frac{f((1-t)a_i + tc_i) - f((1-t)b_i + td_i)}{\Vert (1-t)(a_i-b_i) + t(c_i-d_i) \Vert} \in \mathbb{R}.
$$
Now set $N\in\mathbb{Z}_+$, $N>N_0$. From the second step we have
\begin{align*}
g_N(0) = \frac{f(a_N)-f(b_N)}{\Vert a_N-b_N \Vert} &\leq \frac{\langle v_1,x \rangle}{\Vert x \Vert} + \frac{1}{2^{N-1}} \leq \frac{\langle v,x \rangle}{\Vert x \Vert}\\
&\leq \frac{\langle v_2,x \rangle}{\Vert x \Vert} - \frac{1}{2^{N-1}} \leq \frac{f(c_N)-f(d_N)}{\Vert c_N-d_N \Vert} = g_N(1).    
\end{align*}
So we can find a mean point $t_N \in [0,1]$ such that 
$$
g_N(t_N) = \frac{\langle v,x \rangle}{\Vert x \Vert}.
$$
Now it is enough to take $e_{N-N_0} = (1-t_N)a_N + t_Nc_N$, $f_{N-N_0} = (1-t_N)b_N + t_Nd_N$. This bisequence fulfills our desired conditions.

Thus we have $\partial^{\circ}f(a) = \widehat{B}_a(f).$ Conversely, assume that 
$$
(x, \langle v,x \rangle) \in B_{a}(f), \forall x \in \mathbb{R}^n
$$
for some $v \in \mathbb{R}^n$. Fix $x_0 \in \mathbb{R}^n$, $x_0 \neq 0$. Without any loss of generality, let $\Vert x_0 \Vert = 1$. Then there exists a bisequence $\lbrace a_n, b_n \rbrace_{n=1}^{+\infty}$ such that
$$
a_n, b_n \rightarrow a,\ \frac{a_n - b_n}{\Vert a_n - b_n \Vert} \rightarrow x_0,\ \frac{f(a_n) - f(b_n)}{\Vert a_n - b_n \Vert} \rightarrow \langle v,x_0 \rangle\ (n \rightarrow +\infty).
$$
Let $v_n = a_n - b_n$. Then $\Vert v_n \Vert \rightarrow 0_+$ and $\frac{v_n}{\Vert v_n \Vert} \rightarrow x_0$. Moreover,
\begin{eqnarray*}
\langle v,x_0 \rangle &=& \lim_{n \rightarrow +\infty} \frac{f(a_n) - f(b_n)}{\Vert a_n - b_n \Vert}\\
&=& \lim_{n \rightarrow +\infty} \frac{f(b_n + v_n) - f(b_n)}{\Vert v_n \Vert}\\
&=& \lim_{n \rightarrow +\infty} \frac{f(b_n + \Vert v_n \Vert x_0) - f(b_n)}{\Vert v_n \Vert} + \lim_{n \rightarrow +\infty} \frac{f(b_n + v_n) - f(b_n + \Vert v_n \Vert x_0)}{\Vert v_n \Vert}.
\end{eqnarray*}
For the first component of the above sum, we have an approximation:
$$
\lim_{n \rightarrow +\infty} \frac{f(b_n + \Vert v_n \Vert x_0) - f(b_n)}{\Vert v_n \Vert} \leq \limsup_{b \rightarrow a, t \rightarrow 0_+} \frac{f(b + tx_0) - f(b)}{t} = f^{\circ}(a;x_0).
$$
For the second component, by the Lipschitz condition:
\begin{eqnarray*}
\lim_{n \rightarrow +\infty} \frac{\vert f(b_n + v_n) - f(b_n + \Vert v_n \Vert x_0) \vert}{\Vert v_n \Vert}
&\leq& \lim_{n \rightarrow +\infty} \frac{L \Vert (b_n + v_n) - (b_n + \Vert v_n \Vert x_0) \Vert}{\Vert v_n \Vert}\\
&=& \lim_{n \rightarrow +\infty} \frac{L \Vert v_n - \Vert v_n \Vert x_0 \Vert}{\Vert v_n \Vert}\\
&=& \lim_{n \rightarrow +\infty} L \left\Vert \frac{v_n}{\Vert v_n \Vert} - x_0 \right\Vert\ =\ 0,
\end{eqnarray*}
where $L$ denotes the Lipschitz constant of $f$. Finally, $\langle v,x_0 \rangle \leq f^{\circ}(a;x_0)$, which by the arbitrariness of $x_0$ yields $v \in \partial f(a)$.
\end{proof}

As a corollary from the above theorem, we obtain a very elegant geometric interpretation of the Lebourg Theorem.

\begin{wniosek}
Let $U \subset \mathbb{R}^n$ be open and let $x,y \in U$ be such that $[x,y] \subset U$. Then for a lipschitz function $f:U \rightarrow \mathbb{R}$ there exists a point $c\in (x,y)$ such that $f(y)-f(x) \in \langle \widehat{B}_c(f), y-x \rangle.$
\end{wniosek}

\end{document}